\documentclass[11pt]{amsart}
\pdfoutput=1
\usepackage{amsmath}
\usepackage{amsthm}
\usepackage{amsfonts}
\usepackage{amssymb}
\usepackage{verbatim}
\usepackage{graphicx}
\usepackage[margin=1.0in]{geometry}
\usepackage{url}
\usepackage{enumerate}
\usepackage[pdftex,bookmarks=true]{hyperref}

\usepackage{algorithm}
\usepackage[noend]{algpseudocode}

\newcommand\R{{\mathbb{R}}}

\newcommand\I{{\mathbf{I}}}
\renewcommand\P{{\mathbf{P}}}
\newcommand\Prob{{\mathbf{P}}}
\newcommand\E{{\mathbf{E}}}

\newcommand\Var{\mathbf{Var}}

\newcommand\BG{{\mathbf G}}

\newcommand\BJ{{\mathbf J}}

\newcommand\CB{{\mathcal B}}
\newcommand\CC{{\mathcal C}}

\newcommand\CM{{\mathcal M}}
\newcommand\CN{{\mathcal N}}

\newcommand\CP{{\mathcal P}}

\newcommand\CS{{\mathcal S}}

\newcommand*{\QED}{\hfill\ensuremath{\square}}

\def\x{{\bf X}}
\def\y{{\bf Y}}
\theoremstyle{plain}
  \newtheorem{theorem}{Theorem}[section]
  \newtheorem{conjecture}[theorem]{Conjecture}

  \newtheorem{lemma}[theorem]{Lemma}

\theoremstyle{definition}
  \newtheorem{definition}[theorem]{Definition}

\numberwithin{equation}{section}

\include{psfig}

\usepackage{amsaddr}
\title[Dictionary learning with few samples]{Dictionary learning with few samples and matrix concentration}
\author{Kyle Luh}
\address{Department of Mathematics, Yale University}
\email{kyle.luh@yale.edu}
\author{Van Vu}
\address{Department of Mathematics, Yale University}
\email{van.vu@yale.edu}

\begin{document}

\begin{abstract} \vskip2mm 

%The study of this problem is directly motivated by an application. 
Let $A$  be an $n \times n$ matrix, $X$  be an  $n \times p$  matrix and $Y = AX$. A challenging and important problem in data analysis, motivated by dictionary learning
and other practical problems, is to recover both  $A $ and $X$, given 
$Y$. Under normal circumstances, it is clear that this problem is underdetermined. However, in the case when $X$ is sparse and random, 
Spielman, Wang and Wright  showed that  one  can recover both  $A$ and $X$ efficiently from $Y$ with high probability, given that $p$ (the number of samples)  is sufficiently large. Their method works for 
$p \ge C n^2 \log^ 2 n$ and they conjectured that $p \ge C n \log n$  suffices. The bound $n \log n$ is sharp for an obvious information theoretical reason. 

In this paper, we show that $p \ge C n \log^4 n$ suffices, matching the conjectural bound up to a polylogarithmic factor. 
The core of our proof is a theorem concerning $l_1$ concentration of random matrices, which is of independent interest.

\vskip2mm Our  proof of the concentration result is based  on two  ideas. 
The first is an economical way to apply the union bound. The second is a refined version of Bernstein's concentration inequality for the sum of independent variables.  
Both have nothing to do with random matrices and are  applicable in  general settings.

\end{abstract}
\maketitle

\newpage

%\begin{keywords}
%Dictionary learning, matrix concentration
%\end{keywords}

\section{Introduction}  \label{learning} 
Let   $A$ be an $n \times n$ invertible matrix  and  $X$ be an $n \times p$ matrix; set   $Y :=AX$. The aim of this paper is to study the following recovery problem: 

\centerline {\it  Given $Y$, reconstruct  $A$ and $X$. }

 \noindent  It is clear that in the equation 

\begin{equation} \label{base} Y = AX, \end{equation}   we have $n^2 + np$ unknowns (the entries of $A$ and $X$), and only $np $ equations (given by the entries of $Y$). 
Thus, the problem is underdetermined and one cannot hope for a
unique solution.
However, in practice, $X$ is frequently a sparse matrix. 
 If $X$ is sparse, the number of unknowns decreases dramatically, as the majority of entries of $X$ are zero.  The name of the game here is to find the minimum value of  $p$, the number of observations, which guarantees 
 a unique recovery (e.g. \cite{aharon2006uniqueness} and \cite{georgiev2004blind}).

 One  real-life application that motivates the studies of this problem  is dictionary learning.
 The matrix $A$ can be seen as a hidden dictionary, with its columns being the words. $X$ is a sparse sample matrix. This means that in the columns of $Y$ we observe linear combinations of a few columns of $A$. 
 From these observations, we would like to recover the dictionary. 
   An archetypal example is facial recognition \cite{zhang2010discriminative} \cite{li2013discriminative}. A database of observed 
   faces is used to generate the dictionary and once the dictionary is found, the problem of storing and transmitting 
  facial images can be done very efficiently, as all one needs is to store and transmit  few coefficients.  
  In fact, such dictionary-learning techniques can be utilized to recognize
  faces that are partially occluded 
  or corrupted with noise \cite{wright2009robust}.  
 For more discussion and real-life examples,  we refer to  \cite{kreutz2003dictionary}, \cite{olshausen1996emergence} and the references therein. 
 Another practical situation in which the recovery problem  appears essential is  blind source separation and we refer the reader to \cite{zibulevsky2000blind} for more details.

There have been
 many approaches to efficient recovery
 beginning with the work of \cite{olshausen1996emergence}.  
 Let us mention, among others,   online dictionary learning by
 \cite{mairal2009online}, 
 SIV \cite{gottlieb2010matrix}, the
 relative Newton method for source separation
 by \cite{zibulevsky2003blind}, 
 the Method of Optimal Directions by 
 \cite{engan1999method}, K-SVD in \cite{aharon2006img}, and scalable variants
 in \cite{mairal2009online}.  
   
 While  various  different approaches have been considered, 
 there have not been many rigorous results concerning performance.  
 The first such result has been obtained by Spielman, Wang and Wright  \cite{SWW} concerning recovery with random samples; in other words, $X$ is a random sparse matrix. 
  Before stating their  result, we 
need to discuss the meaning of \textit{ unique} and the random model.   First, notice that if $Y =AX$, then $Y = (AV) (V^{-1} X) $ for any diagonal matrix $V$ with non-zero diagonal entries. 
Furthermore, one can freely permute the columns of $A$ and the rows of $X$ accordingly while keeping  $Y$ the same.  In the rest of the paper, unique recovery will be understood modulo these two operations.

To model $X$, one considers  random Bernoulli-subgaussian matrices, defined as follows: $X$ is a matrix of size $n \times p$ with iid entries $x_{ij}$, where 

\begin{equation}  \label{defX} x_{ij} := \chi_{ij} \xi_{ij} , \end{equation}  where  $\chi_{ij}$ are iid indicator random variables with $\Prob( \chi_{ij} ) =\theta$ and $\xi_{ij}$ are iid random variables with mean 0, variance bounded by 1, 
$$ \E |\xi| \in [1/10,1],$$  and $$\Prob(|\xi| \ge t)  \le 2 \exp (- t^{2} /2) .$$  
 This model includes many important distributions such as the standard Gaussians and Rademachers.
The $1/10$ is introduced for convenience of analysis and not critical to the argument.

Spielman et. al. proved 

\begin{theorem}  \label{theorem:upper} There are constants  $C>0, C'>0$ such that the following holds. 
Let $A$ be an invertible $n \times n$ matrix and $X$ a sparse random $n \times p$  matrix with  $2/n \le \theta \le C'/\sqrt n $
and $\xi_{ij}$ having a symmetric distribution.  Then for  $p \ge C n ^2 \log^2 n$, one can efficiently  find  a solution with probability $1-o(1)$.  
\end{theorem} 

Here and later, efficient means polynomial time. The algorithm designed for this purpose is called ER-SpUD, whose main subroutine is  $l_1$ optimization. We are going to present and discuss this algorithm in Section \ref{algorithms}. 
In the dictionary learning problem, $p$ is the number of measurements, and it is important to optimize its value. From below, it is easy to see that we must have 
$p \ge c n \log n$ for some constant $c >0$. Indeed, if $\theta = 2/n$ (or $c'/n$ for any constant $c'$) and $p < c n \log n$ for a sufficiently small  constant $c$, then the coupon collector 
argument shows that with probability $1-o(1)$, $X$ has an all-zero row. In this case, changing the corresponding column of $A$ will not effect $Y$, and an unique recovery is hopeless. 
Spielman et. al. conjecture

\begin{conjecture} \label{optimal} 
There are constants  $C>0,\alpha>0$ such that the following holds. 
Let $A$ be an invertible $n \times n$ matrix and $X$ a sparse random $n \times p$  matrix with  $2/n \le \theta \le \alpha/\sqrt n $. Then for  $p \ge C n \log n$, one can efficiently  find  a solution with probability $1-o(1)$.  
\end{conjecture}

As a matter of fact, they believe that ER-SpUD should perform well as long as $p \ge C n \log n $, for some large constant $C$. 
They also proved that if one does not cared about the running time of the algorithm, then $p \ge C n \log n$ suffices. 

The analysis in \cite{SWW} boils down to the concentration problem. 
For a vector $v \in \R^n$, let $\mu_v := \E \| X^T v\|_1$.
Let $c$ be a small positive constant  ($c=.1$ suffices) and let $Bad (v)$ be the event that $| \|X^T v\|_1 -\mu _v \| \ge c \mu_v $. We want to  have 

\begin{equation} \label{badevent} \P ( \cup_{v \in \R^n} Bad (v) ) = o(1).\end{equation}   In  other words, with high probability,  $\| X^T v \|_1$ does not deviate significantly from its mean,
simultaneously for all $v \in \R^n$. 

One needs   to find the smallest value of $p$ which guarantees \eqref{badevent}.  Notice that $\|X^T v \|$ is the sum of $p$ iid random variables $|X_i v|$ where $X_i$ are the rows of $X$.
Thus, intuitively   the larger $p$ is,  the more $\| X^T v \|$ concentrates.  From below, we observe that \eqref{badevent} fails if $p \le n-1$, since in this case for any 
matrix $X$ one can find a $v$ such that  $X^Tv =0$ and $\mu_v \ge 1$ (we can take $v$ arbitrarily long).  Spielman, Wang, and Wright  \cite{SWW} 
showed that $p \ge Cn^2 \log^2 n$ suffices.  We will prove

\begin{theorem} \label{theorem:main} 
For any constant $c >0$ there is a constant $C>0$ such that \eqref{badevent} holds for any $p \ge C n \log^{4} n $. 
\end{theorem}

Beyond the current application, 
Theorem \ref{theorem:main} may be of independent interest for several  reasons. While concentration inequalities for random matrices are abundant, 
most of them concern the spectral  or $l_2$ norm. We have not seen one which addresses the $l_1$ norm as in this theorem.
As sparsity  plays crucial role in data analysis, techniques involving $l_1$ norm (such as  $l_1$ optimization) become more and more important. Furthermore,  in the proof we introduce two
general  ideas, which seem to be applicable in many settings. The first is an economical way to apply the union bound and the  second is a refined version of Bernstein's concentration inequality for sums of independent variables.  

%Finally, our study is directly motivated by a matrix recovery  problem of  fundamental interest  in data analysis, which contains dictionary learning as a special case. 
%the discussed in the next  section, and Theorem \ref{theorem:upper} holds for any $p$ which guarantees \eqref{badevent} (see Section \ref{algorithms} for more details). 
Using Theorem \ref{theorem:main}, we  are able to give an improved  analysis of ER-SpUD,  which yields 

\begin{theorem}  \label{theorem:main1} There are constants  $C>0, C'>0$ such that the following holds. 
Let $A$ be an invertible $n \times n$ matrix and $X$ a sparse random $n \times p$  matrix with  $2/n \le \theta \le C'/\sqrt n $.
  Then for  $p \ge C n  \log^{4} n$, one can efficiently  find  a solution with probability $1-o(1)$.  
\end{theorem} 

Our $p$ is  within a $\log^{3} n$ factor from the  bound in Conjecture \ref{optimal} . Furthermore,   we can drop the assumption that  $\xi_{ij}$ are symmetric from 
Theorem \ref{theorem:upper}.

Next, we will be able to refine Theorem \ref{theorem:main} in two ways. 
First,
combining the proof of Theorem \ref{theorem:main1} with a result from random matrix theory, we obtain the following more general result, which handles the case when $A$ is rectangular

\begin{theorem}  \label{theorem:main2} 
There are constants $C, \alpha >0$ such that the following holds. 
Let $n > m$ and $A$ be an $n \times m$ matrix of rank $m$ and  and $X$ a sparse random $m  \times p$  matrix with  $2/n \le \theta \le \alpha/\sqrt n $. Then for  $p \ge C n \log^{4} n$, one can efficiently find a 
solution with probability $1-o(1)$ \end{theorem}

Second, in  the sparest case $\theta := \Theta (1/n)$, we develop a new algorithm that  obtains the optimal bound $p = C n \log n$,  proving Conjecture \ref{optimal} in this regime.

\begin{theorem}  \label{theorem:main3}  For any  $c >0$  there is a constant $C >0$ such that the following holds. 
Let $A$ be an invertible $n \times n$ matrix and $X$ a sparse random $n \times p$  matrix with  $\theta = c/n $. Then for  $p \ge C n \log n$, one can efficiently find a 
solution with probability $1-o(1)$ \end{theorem}

%The assumption that $A$ is invertible appears critical.  Spielman et. al. asked whether one can do the recovery in 

Finally, let us mention the issue of theoretical recovery, regardless the running time. Without the complexity issue, Spielman et. al. showed that 
$p \ge C n \log n$ suffices, given that the random variable $\xi_{ij}$ in the definition of $X$ has a symmetric distribution. We could strengthen this theorem by removing this assumption.

\begin{theorem}  \label{theorem:main4} There are constants  $C>0, C'>0$ such that the following holds. 
Let $A$ be an invertible $n \times n$ matrix and $X$ a sparse random $n \times p$  matrix with  $2/n \le \theta \le C'/\sqrt n $.
  Then for  $p \ge C n  \log n$, one can   find  a solution with probability $1-o(1)$.  
\end{theorem}

The rest of the paper is organized as follows. In Section \ref{outline}, we present 
the main ideas behind the proof of Theorem \ref{theorem:main}. The details follows next in Section \ref{details}. Section \ref{algorithms} 
contains the accompanying algorithms and an improved 
 analysis of ER-SpUD, following \cite{SWW}.   Section \ref{appendix:rectangular} addresses a generalization to rectangular dictionaries.  
Section \ref{verysparse} introduces a new algorithm that achieves the optimal bound in the sparse regime. In Section \ref{theoreticalbound}, we prove Theorem \ref{theorem:main4}. 
  We conclude with  Section \ref{numerical},  in which we  present some numerical 
experiments of the various algorithms. 

\vskip2mm 

{\it Acknowledgement.} We would like to thank D. Spielman for bringing the  problem to our attention.

\section{The main ideas and lemmas} \label{outline}
\subsection{The standard $\epsilon$-net argument} 

Let us recall our task.    For a vector $v \in \R^n$, let $\mu_v := \E \| X^T v\|_1$.
Let $c$ be a small positive constant  ($c=.1$ suffices) and let $Bad (v)$ be the event that $| \|X^T v\|_1 -\mu _v \| \ge c \mu_v $. We 
want to show that if $p$ is sufficiently large, then 

\begin{equation} \label{badevent1} \P ( \cup_{v \in \R^n} Bad (v) ) = o(1).\end{equation}  

For the sake of presentation, let us assume that the random variables $\xi_{ij}$ are Rademacher (taking values $\pm 1$ with probability $1/2$); the entries $x_{ij}$ of 
$X$ have the form $x_{ij} = \chi_{\ij} \xi_{ij} $, where $\chi_{ij}$ are iid indicator variables with mean $\theta$. 
We start by a quick proof of the bound $p \ge C n^2 \log^2 n$ obtained in \cite{SWW}. Notice that the union in \eqref{badevent} contains infinitely many 
terms. The standard way to handle this is to use an $\epsilon$-net argument.

\begin{definition}  A set $\CN \subset \R^n $ is an $\epsilon$-net of a set $D \subset \R^n $ in  $l_q$ norm, for some $ 0 < q \le \infty$, if for any  $x \in D$ there is $y \in \CN$ so that 
$\| x- y\| _q \le \epsilon$. The unit sphere in $l_q$ norm  consists of vectors $v$ where $\| v\|_{q} = 1$. $B$ denotes the unit sphere in $l_1$ norm. \end{definition}

Considering the vectors in $B$ is sufficient to prove the result.
It is easy to show that for any $v \in B$ 

$$ \mu_{min} := p \sqrt {\theta/ n} \le \mu_v  \le p \theta := \mu_{\max} , $$ where the lower bounds attend at $v=\frac{1}{n} {\bf 1} $  ({\bf 1} is the all one vector) 
and the upper bound at  $v=(1, 0, \dots, 0)$.  Let $\CN_0$ be the set of all vectors in $B$ whose coordinates are integer multiples of $n^{-3}$. 
Any vector in $B$ would be of distance at most $n^{-2}$ in $l_1$ norm from some vector in $\CN_0$ (thus $\CN_0$ is an $n^{-2}$-net of $B$). 
A short consideration shows that if $u, v \in B$ are within $n^{-2}$ of each other, then 

$$|\mu_v - \mu_u |  = o( \mu _{min } ). $$

Thus, to prove \eqref{badevent}, it suffices to show that 

\begin{equation} \label{badevent2} \P ( \cup_{v \in \CN_0} Bad (v) ) = o(1).\end{equation}  

In order to bound   $\P ( \cup_{v \in \CN_0} Bad (v) )$, let us first bound $\P(Bad (v))$ for any $B$. Notice that

$$\| X^T v \|_1 = \sum_{i=1} ^ p  | X_i v |, $$ where $X_i$ are the columns of $X$. The random variables $|X_i v|$ are iid, and one is poised to apply another standard tool, Bernstein's inequality 
for the sum of independent random variables.

\begin{lemma}  \label{Bennett} 
Let $Z_1, \dots, Z_n$ be independent  random variables such that $|Z_i| \le \tau$ with probability 1. 
 Let $S:=\sum_{i=1}^n Z_i$. Then  for any $T >0$

$$\max\{ \P( S- \E S \le -T) , \P ( S - \E S| \ge T) \}   \le  \exp( - \frac{T^2}{2( \Var S  + T \tau ) })  \le \ \exp( - \min \{ \frac{T^2}{ 4\Var S } , \frac{T}{4 \tau }\} ) .$$

\end{lemma}

In our case $Z_i = |X_i v| = \sum_{i=1}^n \x_{ij} v_j $.  As $|x_{ij} = \chi_{ij} \xi_{ij} | \le 1$ with probability 1 (we assume that $\xi_{ij}$ are Rademacher)

$$|Z_i |   \le \sum_{i=1}^n |v_j|  =\|v \|_1 =1 $$ with probability 1. This means we can set $\tau=1$. Furthermore 

$$\Var \sum_{i=1}^p Z_i = p \Var Z_i  \le p \E |X_i v| ^2 = p \sum_{j=1}^n \theta v_j^2 \le p \theta \sum_{j=1}^n |v_j| = p \theta. $$  Finally, one can set 
$T = c \mu_{min} = c p \sqrt {\theta/n} $.  Lemma \ref{Bennett} implies that 

$$\P(Bad(v) \le  2 \exp(- \min \{ \frac{c^2 p^2 \theta/n }{ 4 p \theta }, \frac{ c p \sqrt {\theta/n} }{4} \} ) = 2 \exp( - \frac{c^2 p }{4n } ) $$ since $\sqrt {\theta/n } \ge 1/n$ as 
$\theta \ge 1/n$.  

Using the union bound 

\begin{equation} \P( \cup_{v \in \CN_0} Bad (v) \le \sum_{v \in \CN_0 } \P(Bad(v))  \end{equation} we obtain 

$$ \P( \cup_{v \in \CN_0} Bad (v) ) \le |\CN_0 |  \times 2  \exp( - \frac{c^2 p }{4n } ) . $$ It is easy to check that $\CN_0 = \exp(\Omega (n \log n) )$. So,
in order to make the RHS $o(1)$, we need $p \ge C n^2 \log n$ for a sufficiently large constant $C$.  For the case when $\xi_{ij}$ are not Bernoulli (but still subgaussian) 
the calculation in \cite{SWW} requires  an extra logarithm term, which results in the bound $p \ge C n^2 \log^2 n $.

\subsection {New ingredients} 

Our first  idea is to find a  more efficient variant of the union bound

$$\P (\cup_{v \in \CN_0}  Bad(v) ) \le \sum_{v \in \CN_0 } \P ( Bad (v) ).$$

Motivated by the inclusion-exclusion formula
we try to capture some gain when 
$\Prob ( Bad (u) \cap  Bad (v)) $ is large for many pairs $u, v$. We  observe that if we can group the elements $v$ of the net  into clusters so that within each cluster, the events 
$Bad (v) $ (seen as subsets  of the underlying probability space) are close to each other. Assume, for a moment, that one can split the net  $\CN_0$ into $m$ disjoint  clusters $\CC_i$, $1 \le i \le m$, so that 
if $u$ and $v$ belong to the same cluster  $\P ( Bad(u) \backslash Bad (v) ) \le  p_1$,   where $p_1$ is much smaller than $p_0$, then 

$$ \P (\cup_{ v \in \CC_i} Bad (v) ) \le \P  (Bad (v^{[i]} ) ) +  | \CC_i| p_1 , $$ where $v^{[i]} $ is a representative point in $\CC_i$. Summing over $i$, one obtains 

\begin{equation} \label{2level} \P  (\cup_{v \in \CN_0 } Bad (v)  ) \le \sum_{i=1} ^m  \Prob( \cup_{v \in \CC_i} Bad (v)) \le \sum_{i=1}^m  \Prob (Bad (v^{[i]} ))  + | \CN_0 | p_1  \le m p_0 + | \CN _0 | p_1. \end{equation} 

We  gain significantly if  $p_1$ is much smaller than $ p_0$ and $m$ is much smaller than $ |\CN_0 | $.  Next, viewing the set of representatives $v^{[i]} $ as a new net $\CN_1$, we can iterate  the argument, obtaining the following lemma. 

\begin{lemma}  \label{newnet} Let $\CP$ be a probability space. 
Let $\CN = \CN_0$ be a finite set, where to each element $v \in \CN_0$ we associate a set  $Bad_0 (v) \subset \CP $. Assume that we can construct a  sequence of sets 

$$\CN_L ,  \CN_{L-1} , \dots ,  \CN_0 , $$ and  for each $u \in \CN_l, 1 \le l \le L$  an event $Bad_l(u)$ such that the following holds. 
For  each $v \in \CN_{l-1}  $, there is $u \in \CN_{l} $ such that 
$\P (Bad_{l-1}  (v) \backslash Bad_{l}  (u) ) \le p_{l} $ and  for  each  $u \in \CN_L$, $\P( Bad_L (u) ) \le p_0$. Then

\begin{equation} \label{Llevel} \P (\cup_{v  \in \CN_0} Bad_0 (v)  ) \le  | \CN_L| p_0 + \sum_{l=1} ^{L} | \CN_{l-1} | p_{l} . \end{equation} 
\end{lemma} 

The construction of $\CN_l$ are of critical importance, and we are going to construct them using the $l_{\infty}$ distance, rather than the obvious choice of 
$l_1$. {\it This is the key point of our method. }

The next main technical ingredient is a more efficient way of using Bernstein's inequality, Lemma \ref{Bennett}. 
Recall the bound

\begin{equation} \label{B} \P (| S - \E S| \ge T ) \le 2 \exp( - \frac{T^2}{2( \Var S  + T \tau ) })  \le 2 \exp( - \min \{ \frac{T^2}{ 4\Var S  }, \frac{T}{4 \tau } \}) . \end{equation}

The first term $\frac{T^2}{4 \Var S }$ on the right most formula is usually optimal. However, we need to  improve the second term. The idea is to replace $\tau$ with a smaller quantity 
$\tau'$ such that the probability that $|Z_i| \le \tau'$ is close to 1. Let  us illustrate this  idea with the upper tail. Set $\mu := \E S$, we consider  
$$ \P (S \ge \mu + T). $$

Write 

$$Z_i :=  Z_i \BJ_i  +  Z_i \I_i $$ where $\BJ_i$ is the indicator of the event $|Z_i| \le \tau'$ and $\I_i =1 - \I_i$. Thus
$$ S :=   \sum_i Z_i \I_i + \sum_i Z_i \BJ_i =  Q +  S(1) . $$ Let $\mu_j$ be the expectation of $S (j)$. Then 

$$\P ( S \ge \mu + T) \le \P( Q \ge \mu_1 + T/2) + \P (S (1) \ge  \mu_2 + T/2 ). $$

We can use Lemma \ref{Bennett} to bound $\P (Q \ge \mu_1 + T/2) $, which provides a  bound better than \eqref{B} as now $\tau' < \tau$. 
On the other hand,  if the probability that $|Z_i| \ge \tau'$ is  small, then we can bound $\P (S(1) \ge \mu_2 + T/2 )$ in a different way,  exploiting the fact that there will be very few non-zero summands  in $S (1)$. 

We can (and have to)  further refine this idea by considering a sequence of $\tau'$, breaking $S$ into the sum of $Q$ and $S(k), 1\le k \le M$, for a properly chosen $M$. 
This will be our leading idea   to bound the difference probability $p_l$ in the next section.

On the abstract level, our method bears a similarity to the chaining argument from the theory of Banach spaces. We are going to discuss this point in Section \ref{remark}.

%\end{document} 

\section{Proof of Theorem \ref{theorem:main}} \label{details} 

For the sake of presentation, we  assume that $x_{ij} = \chi_{ij} \xi_{ij}$ where $\chi_{ij}$ are iid Bernoulli random variables with mean $\theta$ and $\xi_{ij}$ are iid Rademachers random variables.
In fact, $p \geq C n \log^3 n$ is sufficient for the Rademacher case. 
The proof can be easily modified  for $\xi_{ij} $ being general sub-gaussian at the cost of a $\sqrt{\log n}$ factor  in the bound for $p$ (See Section \ref{extension}). We recall the notation $\mu_{min } = p \sqrt {\theta/n}, 
\mu_{max} = p \theta$; $\mu_v := \E \| X^T v \| _1 $.  $B$ is the set of all vectors of unit $l_1$ norm.

We set $p = C n \log ^3 n $, for a sufficiently large constant $C$. Let $T :=   \frac{ c_0  \mu _{min}} {\log n } $ for a small constant $c_0 >0$ and 
$K :=  \lceil \frac{ 6\mu_{max} }{ T } \rceil$.

\subsection{$\alpha$-nets in $l_{\infty} $ norm } \label{nets}

\begin{lemma}  \label{smallnet} 
For any $1 \ge \alpha \ge 2/n $, $B$ admits an $\alpha$-net in  $l_{\infty} $ norm of size at most $\exp( 2 \alpha^{-1} \log n )$. 
\end{lemma}

\begin{proof} Let $\CN$ be the collection of all vectors $v \in B$, whose coordinates  are integer multiples of $\alpha$. Obviously,  $\CN$ is an $\alpha$-net of $B$ 
in $l_{\infty}$ norm. Furthermore, any $v \in \CN$ satisfies $\|v\|_1 \le 1$, so it has at most $k:=  \alpha^{-1} $ non-zero coordinates. If a coordinate is non-zero, it can take at most $2\alpha^{-1} +1 \le 3k$ values. 
Therefore, 

$$| \CN | \le \sum_{i=0}^k {n \choose i} (3k )^k . $$ 

As $\alpha \ge 2/n$, the RHS is at most 

$$n {n \choose k } (3k)^k  \le n (\frac{en }{k} \times 3k) ^{k} = n ( 2en )^{k}  \le \exp( 2 \alpha^{-1} \log n ) . $$ 
 \end{proof}

The key here is that we consider an $\alpha$-net in $l_{\infty}$ norm, rather than in  $l_1$ norm, which appears to be a natural choice.  

\subsection{ Building a  nested sequence}  \label{sequence} 

Recall that   $\CN_0$  is the set of vectors $v$  in $B$ whose coordinates are integer multiples of $n^{-3}$.  We have 

\begin{equation} \label{boundN0} |\CN_0 |  \le (2n^3 +1) ^n \le \exp( 4 n \log n ). \end{equation} 

Consider the sequence $\alpha_0 =2/n; \alpha_l =2 \alpha_{l-1}$ for $l=1, \dots, L$, where $L \le \log_2 n $ is the first index such that $\alpha_L > 1/2$. 
Let  $\CN_l'$  be an $\alpha_l$-net of $B$ in the $l_{\infty}$ norm. By Lemma \ref{smallnet}, we can choose $\CN_l'$ such that  

\begin{equation} \label{boundNi1}  |\CN_l ' |  \le \exp( 2 \alpha_l^{-1} \log n ).  \end{equation}

We now build a nested sequence $\CN_L \subset \CN_{L-1} \subset \dots \subset \CN_1 \subset \CN_0$ as follows. Assume that $\CN_{l-1}$ has been built. Use the points in $\CN_l'$ as centers to 
construct a Voronoi partition of the points of $\CN_{l- 1}$  with respect to the $l_{\infty} $ norm (ties are broken arbitrarily). For each point $u \in \CN_l'$, let $C_u$ be the subset of $\CN_{l-1}$ corresponds to  $u$. 
By definition,  $\| u-v\| _{\infty} \le \alpha_l$  for any $v \in C_u$,

Partition the interval $[\mu_{min}, \mu_{max}]= [p \sqrt{ \theta/n}, p \theta ]$ into $K$  intervals $I_1, \dots, I_K$ of equal lengths. We partition $C_u$ further into $K$ subsets $C_{u,j}, 1 \le j \le K$, where 
$v \in C_{u,j}$ if $\E \| Xv \| _1 \in I_j $. By this construction, if $v, w$ belong to the same $C_{u,j}$, then  by the definition of $K$, we have the key relations 

\begin{equation} \label{keybounds}  \| v-w\| _{\infty} \le 2  \alpha_l \,\,\, {\rm and} \,\,\,  | \E \| X v \| _1 - \E \| X w \| _1 | \le p\theta /K \le T /6 . \end{equation}

From  each set $C_{u,j}$ choose an arbitrary element $v$.  Thus,  each $u \in \CN_l'$ gives rise to a   set $R_u$ of $K$ elements ($R$ stands for representative). Define 

$$\CN_l  := \cup_{u \in \CN_l' }  R_u. $$ 

It is clear that $\CN_l \subset \CN_{l-1} $ and 

\begin{equation} \label{boundNi2}  |\CN_l |  \le K | \CN_l' |  \le K \exp( 2 \alpha_l^{-1} \log n ) . \end{equation}

\subsection{Bounding the differences} \label{difference} 

Consider  the construction of $\CN_l$, $1 \le l \le L$, from Section \ref{nets}.  Let $v \in \CN_l$. Thus, $v \in C_{u,j}$ for some $u \in \CN_l'$ and $1 \le j \le K$. 
Consider another point $w \in \CN_{u, j}$. Our main task is to show

\begin{lemma} \label{hardlemma} For all pairs $v,w$ as above
\begin{equation} \label{rhobound} \rho(v, w) :=  \P (| \| X^T v\| _1 - \| X^T  w\| _1| \ge  T ) \le \exp(-  5 \alpha_l^{-1} \log n) . \end{equation}   \end{lemma}

The rest of this section is devoted to the proof of this lemma. By \eqref{keybounds}, we have 

\begin{equation} \label{keybounds1}  \| v-w\| _{\infty} \le 2  \alpha_l \,\,\, {\rm and} \,\,\,  | \E \| X ^T v \| _1 - \E \| X ^T w \| _1 | \le p\theta /K \le T/6 . \end{equation} 

%In the argument below, we fix the level $i$ and omit the subindex $i$ from $\CN_i, \CN'_i, \alpha_i$. 

\noindent Define $Z_I = |X_i v| - |X_i w|$, where $X_i$ is the $i$th row of $X^T$; we have 

$$ \|  X^T v \| _1 - \| X^T  w \|_1 =  \sum_{i=1}^p   (|X_i v| - |X_i w| )   =\sum_{i=1}^p Z_i .$$  Set $S:= \sum_{i=1}^p Z_i$; by symmetry, it suffices to bound 

$$\P (Z_1 + \dots  + Z_p \ge  T ) :=   \P( S \ge T ). $$

\noindent  Notice that by the triangle inequality  $$|Z_i| = \Big| |X_iv | - |X_i w| \Big| \le | X_i (v-w) |. $$  Therefore, 

$$\Var Z_i  \le \E Z_i^2 \le \E |X_i (v-w) |^2  = \theta \sum_{j=1}^n (v_j -w_j)^2 . $$

\noindent Recall that  $\|v\|, \|w\| \le 1$ and $\| v-w\|_{\infty} \le \alpha_l$.  Therefore 

$$\sum_{j=1}^n (v_j -w_j)^2 \le \alpha_l \sum_{j=1}^n |v_j| + |w_j| = 2 \alpha_l . $$

\noindent This implies 

\begin{equation} \label{variance}  \Var Z_i \le \E Z_i^2 \le 2 \alpha_l \theta . \end{equation} 

We denote by $\I_{i,k}$ the event that $\tau_k < Z_i \le \tau_{k-1}$ for $k=1, \dots, M$ and 
$J_i$ the event that $|Z_i| \le \tau_M$, for a sequence $\tau_k$, $k=0, \dots, M$, where 
$\tau_0 =2$;  $\tau_i =2^{-i} \tau_0$ and $M $ is the first index so that 

\begin{equation} \label{Mbound} \min \{ \frac{\tau_M^2}{ 8 \alpha_l \theta }, \frac{\tau_M}{ 4 \alpha_l } \} \ge 8 \log n. \end{equation} 

Note  that  if $\alpha_l \le \frac{1}{32} \log^{-1} n $ then such an index $M \ge 1$ exists. We will proceed with this assumption
and cover the remaining cases  at the end of the proof. 
Apparently,

$$Z_i \le \sum_{i=1}^k Z_i \I_{i,k} + Z_i J_i. $$ Set $S(k) = \sum_{i=1}^p Z_i \I_{i,k} $ for $k=1, \dots, M$ and 
$Q = \sum_{i=1}^p Z_i \BJ_i $. We have

$$\P (S \ge T) \le \P (Q \ge T/2) + \sum_{k=1}^M \P( S(k) \ge \frac{T}{2M} ). $$

To bound $\P(Q \ge T/2)$, we notice that (see \eqref{rhobound})  the choice of  $\tau_M$ guarantees that  
$\P (\BJ_i)  \ge 1 - 2n^{-8} $ for all $i=1,\dots, p$. As $|Z_i| \le 2$ with probability 1, it follows that 

$$| \E Z_i \BJ_i - \E Z_i |  \le 4 n^{-8} $$ and so 

$$| \E Q - \E S | \le  4p n^{-8} = o( n^{-6}), $$ as $p= o(n^2)$.
On the other hand, by \eqref{keybounds1},   $T \ge 5 (\E S + n^{-6}) $. Thus 

$$\P (Q \ge T/2 ) \le \P( Q \ge \E Q +  T/4 ). $$  By definition, $Q$ is sum of $p$ iid random variables, each is  bounded by $\tau_M$ in absolute value with probability 1. Furthermore, 
by \eqref{variance}

$$\Var Q = p \Var Z_1 \BJ_1 \le  p \E Z_1^2  \le  2 \alpha_l \theta p . $$ By Lemma \ref{Bennett}, we have 

\begin{equation} \label{Qbound} 
\P( Q \ge \E Q + T/4) \le 2 (\exp (- \min \{ \frac{(T/4)^2}{8\alpha_l \theta p } , \frac{T/4}{ 4 \tau_M } \})  =  2 \exp (- \min \{ \frac{T}{128 \alpha_l \theta p }, \frac{T}{16 \tau_M } \}). \end{equation}

Now we bound  $\P( S(k) \ge \frac{T}{2M} )$, for $k=1, \dots, M$.  Recall that $S(k) := \sum_{i=1}^p Z_i \I_{i,k} $ is a sum of iid non-negative  random variables, each is either 0 or in $(\tau_k $ and $\tau_{k-1}]$. 
Thus, if $S(k) \ge T/2M$ there must be at least $p_k:= \frac{T/2M}{\tau_{k-1} } $ indices $i$ such that $Z_i > \tau_k $.  Let $\rho_k$ be the probability that $Z_1 > \tau_k$. Then by the union bound and the fact that 
$p= o(n^2) $, 

\begin{equation} \label{difference2}  \P( S(k) \ge \frac {T}{2M} ) \le {p \choose {p_k} }  \rho_k ^{p_k}   \le (\frac{e p }{p_k} \rho_k )^{p_k}  \le ( \frac{n^2}{2}  \rho_k)^{p_k} . \end{equation}

%\end{document} 
To complete the analysis,  we need to estimate  $\rho_k$. By definition 

$$\rho_k := \P( |X_1 v| - |X_1 w | > \tau_k ) \le \P( | X_1 ( v-w) | \ge \tau_k ) . $$  The random variable $\tilde Z_1 := X_1( v-w)  = \sum_{j=1}^n \xi_j (v_j -w_j )$ has mean 0. 
Furthermore,  by \eqref{variance}, $\Var \tilde Z_1 \le  \tilde Z_1^2 \le 2 \alpha_l \theta$. Finally, each term 
$\xi_j (v_j- w_j)$ is at most $\alpha_l$ in absolute value. Thus Lemma \ref{Bennett} implies 

\begin{equation} \label{rhobound} \rho_k \le \P( |\tilde Z_1| \ge \tau_k) \le 2  (\exp( - \min \{\frac{\tau_k^2}{ 8 \alpha_l \theta },\frac{\tau_k}{ 4\alpha_l } \} ) . \end{equation} 

This and \eqref{difference2} yield 

\begin{equation} \label{Skbound} \P(S(k) \ge \frac{T}{2M} ) \le  2\exp( - \Big( \min \{ \frac{\tau_k^2}{ 8 \alpha_l \theta }, \frac{\tau_k}{ 4\alpha_l } \} + 2 \log n \Big)  p_k ). \end{equation}  By \eqref{Mbound},

$$ \min \{ \frac{\tau_k^2}{ 8 \alpha_l \theta }, \frac{\tau_k}{ 4\alpha_l } \}  \ge 8 \log n, $$ so 

$$ \Big( \min \{ \frac{\tau_k^2}{ 8 \alpha_l \theta }, \frac{\tau_k}{ 4\alpha_l } \} + 2 \log n \Big)  p_k \ge  \frac{1}{2} \min  \{ \frac{\tau_k^2}{ 8 \alpha_l \theta }p_k,  \frac{\tau_k}{ 4\alpha_l } p_k \}. $$
By definition  $p_k = \frac{T/2M}{\tau_{k-1} } = \frac{T/4M}{\tau_k } $, as $\tau_{k-1} = 2\tau_k$. Therefore, 

$$ \frac{1}{2} \frac{\tau_k^2}{ 8 \alpha_l \theta }p_k = \frac{\tau_k T }{ 64 M \alpha_l \theta } $$ and 

$$\frac{1}{2} \frac{\tau_k}{ 4\alpha_l } p_k = \frac{T}{ 32 M \alpha_l }. $$ 

\noindent By \eqref{Qbound} and \eqref{Skbound}, we conclude that 

\begin{equation} \label{Sbound} \P( S \ge T )  \le  2 \exp (- \min \{ \frac{T^2 }{128 \alpha_l \theta p }, \frac{T}{16 \tau_M } \}) + 
\sum_{k=1}^M 2 \exp( -\min \{  \frac{\tau_k T }{ 64 M \alpha_l \theta }, \frac{T}{ 32 M \alpha_l } \}). \end{equation} 

A routine verification (see Section \ref{magnitude}) shows that once $p \ge C n \log^ 3 n$ for a sufficient large constant $C$, then the RHS in \eqref{Sbound} is at most 
$\exp (- 5 \alpha_l^{-1} \log n)$, completing the proof for the case $\alpha_l  \le \frac{1}{32} \log^{-1} n$. 

To complete the proof, we now treat  the remaining case when $\alpha_l \ge \frac{1}{32} \log^{-1} n$.. In this case, we do not need to split $Z_i$. 
Recall $S= Z_1+\dots +Z_p$ whre $|Z_i| \le 2$ with probability 1, $\E S \le T/6$ and $\Var S \le 2 p \theta \alpha_l$. By Lemma \ref{Bennett}, we have 

$$\P( S \ge T) \le \P( S \ge \E S + T/2)  \le \exp( - \min \{ \frac{T^2}{ 8 p \theta \alpha_l }, \frac{T}{8} \}). $$

By the analysis of \eqref{Sbound}, we already know that   $\frac{T^2}{ 8 p \theta \alpha_l } \ge 5 \alpha_l^{-1} \log n$. On the other hand, as $\alpha_l \ge \frac{1}{32} \log^{-1} n$

$$ \frac{T}{8} = \frac{c_0 p \sqrt {\theta /n} } {8 \log n }  =\frac{c_0 C}{8}  \sqrt {\theta n } \log^2 n  \ge  5 \alpha_l^{-1} \log n, $$ given that $c_0 C$ is sufficiently large. This completes the proof.

\subsection{Proof of the Concentration lemma} 

For $v \in \CN_l, 0 \le l \le L$, let  $Bad_l (v) $ be the event that $| \| X v\| _1 - \mu_v | \ge  2(L+1 -l) T $. 
For $l=0$, $2(L+1-l_ T = 2 (L+1) T \le  \frac{2c_0 (\log_2 n +1)  \mu_{min} }{ \log n }  \le 4 c_0  \mu_{min}$. Thus, 

$$\P (\cup_{ v \in \CN_0 }  |\| X^T v\|_1 - \mu_v | \ge 4 c_0 \mu_{min}  ) \le \P ( \cup_{v \in \CN_0}  Bad_0 (v)). $$

Assume that there is a number $p_0$ such that $\P (Bad_0 (v)) \le p_0$ for all $v \in \CN_0$. Assume furthermore that for any $1 \le l \le L$, there is a number $p_l$ such that 
for $ v \in \CN_l$ and $w \in \CN_{l-1} $ where $v$ is the representative of the set $C_{(u, k)}$ that contains $w$ 
(see the construction in Section \ref{sequence}). 

$$\P ( Bad_l (w) \backslash Bad_{l-1} (v) ) \le p_l. $$

Then by Lemma \ref{newnet}

$$\P (\cup_{v \in \CN_0} ) \le  |\CN_L| p_0 +\sum_{l=1}^{L} | \CN_{l-1}  | p_{l} . $$

To find $p_l$, notice that  if $Bad_{l-1} (w)$ holds and $Bad _l (v)$ does not, then 
$| \| X^T w \|_1  - \mu_w | \ge 2 (L+2 -l) T$ and $| \| X^T v \|_1 - \mu_v | \le 2 (L+1 -l) T$. 
By \eqref{keybounds},  $| \mu_v -\mu_w | \le T$. It  thus follows that 

$$| \| X^T w \| _1  - \| X^T v \| _1 | \ge T . $$

By the main lemma  of  Section \ref{difference}, we know that the probability of this event  is at most 
$p_l:= \exp( -5 \alpha_l^{-1} \log n)$, for all $l$.  Recall from Section \ref{sequence} that 

$$ |\CN_l  | \le K \exp( 2 \alpha_l^{-1} \log n ) =K \exp( 4 \alpha_l ^{-2} \log n ) , $$ we have

$$  \sum_{l=1}^{L}  |\CN_{l-1} | p_l  \le \sum_{l=1}^{L}  \exp ( -4 \alpha_l ^{-1} \log n )  \times K \exp (4 \alpha_{l} ^{-1}  \log n ) .$$ 

Since $K = O(n^{1/2})$ and $ \alpha_l^{-1} \log n \ge \log n $,  the RHS is at most 

$$  \sum_{l=1}^{L'}  \exp ( - .5 \alpha_l ^{-1} \log n )  = o(1). $$

To conclude, notice that by Lemma \ref{Bennett}, we can set $p_0 := 2 \exp(- \min \{ \frac{T^2}{8 p \theta }, \frac{T}{8} \} )$. As $|\CN_L| \le \exp( -2 \alpha_L^{-1} \log n ) \le \exp( 4 \log n ) $ since $\alpha_L \ge 1/2$, we have 

$$p_0 | \CN_0 |  = o(1), $$ as long as $ \min \{ \frac{T^2}{8 p \theta }, \frac{T}{8}\} \ge 5 \log n$. This condition holds if $p \ge C n \log^ 3n $ for a sufficiently large constant $C$.  
This implies that

$$\P (\cup_{ v \in \CN_0 }  \{ \| X^T v - \mu_v \| \ge 4 c_0 \mu_{min}  )  = o(1) ,$$ and we are done by \eqref{badevent2}.

\subsection{The magnitude of $p$ } \label{magnitude} 

We present the routine verification concerning the exponents in \eqref{Sbound}. This is the only place where the magnitude of $p$ matters. 
Recall that 
$T = \frac{c_0  \mu_{min} } {\log n}  = \frac{c_0 p \sqrt { \theta/n }}{ \log n} $ and $p = C n \log^3 n $ (since for the sake of exposition we are only considering the Rademacher case).  We have 

$$\frac{T^2 }{128 \alpha_l \theta p} = \frac{c_0^2  p^2  \theta /n} {128 \theta p \log^2 n } \alpha_l^{-1} = \frac{c_0^2  p}{n} \alpha_l^{-1} = c_0^2  C  \alpha_l^{-1} \log n \ge 4.1 \alpha_l^{-1} \log n, $$ provided that 
 $c_0^2 C \ge 4.1$. 
 
 By  the definition of $M$ in \eqref{Mbound}, we have 
 
 $$ 32  \log n \ge   \min \{ \frac{\tau_M^2}{ 8 \alpha_l \theta } , \frac{\tau_M}{4 \alpha_l } \} \ge 8 \log n. $$ This implies that 
 
 $$\tau_M \le \max \{  16 \sqrt { \alpha_l \theta \log n } , 128 \alpha_l \log n \} . $$ It follows that 
 
 $$ \frac{T}{16 \tau_M } \ge \min \{ \frac{T}{256 \sqrt{\alpha_ \theta \log n }} , \frac{T}{2048 \sqrt 2  \alpha_l \log n } \} .$$

 \noindent By the definition of $p$ and $T$ 
 
$$  \frac{T}{256 \sqrt{\alpha_ \theta \log n }  } =\frac{c_0 p }{ 256 \sqrt { \alpha_l n \log^3 n}} = \alpha_l^{-1} c_0 C \log n  \sqrt { n \alpha_l } \ge 4.1 \alpha_l^{-1} \log n , $$ since 
$c_0 C \ge 4.1 $ and $ n \alpha_l \ge n \alpha_0 \ge n \frac{2}{n} > 1$.  Furthermore,

$$ \frac{T}{2048 \sqrt 2  \alpha_l \log n } = \alpha_ l^{-1} \frac{c_0 C n \log^3 n \sqrt {\theta/n}} { 2048 \sqrt 2 \log n }  =\omega ( \alpha_l^{-1} \log n). $$

\noindent  Next, we bound the exponent $\frac{T}{32 M \alpha_l }$. As $M \le \log n$, we have 

$$\frac{T}{32 M \alpha_l }  \ge \frac{c_0 C n \log^2 n \sqrt{\theta /n } }{32 \alpha_l \log n} = \alpha_l^{-1} \frac{c_0 C}{32} \sqrt {\theta n } \log n \ge 4.1 \alpha_l^{-1} \log n, $$ provided that 
$c_0 C /32  \ge 4.1$, since $\theta n \ge 1$. 

\noindent Finally,  we bound the exponent $\frac{\tau_k T}{64 M \alpha_l \theta } $. By definition $\frac{\tau_k^2} {8 \alpha_l \theta }  \ge 8 \log n$ and $M \le \log n$  thus

$$ \frac{\tau_k T}{64 M \alpha_l \theta } \ge   \frac{8 \sqrt { \alpha_l \theta \log n} T }{ 64 \log n \alpha _l \theta } = \alpha_l^{-1} \frac{c_0 C}{8} \sqrt {n \alpha_l } \log^{3/2} n  =\omega (\alpha_l^{-1} \log n ),  $$
concluding the proof.

 \subsection{Extension from Rademacher to general sub-gaussian variables } \label{extension}
We introduce the truncation operator $T_{\tau}: \mathbb{R}^{n \times p} \rightarrow \mathbb{R}^{n \times p}$ as
\begin{displaymath}
   (T_{\tau}[M])_{ij}= \left\{
     \begin{array}{lr}
       M_{ij} &  |M_{ij}| \leq \tau\\
       0 &  else
     \end{array}
   \right.
\end{displaymath}  
Let $\tau = \sqrt{C \log n}$ and let 
$$
X' = T_{\tau}[X].
$$

For $C$ sufficiently large, the probability that $X'= X$ is $1- o(1) $. This allows us to work with random matrix whose entries are bounded by 
$\tau$ (instead of 1 as in the Rademacher case).  The same proof will go through if we increase $p$ by $C_1 \tau$, for a sufficiently large constant $C_1$.
This means $p= O( n \log^{3.5} n)$ suffices. We round $3.5$ up to 4 for cosmetic reasons. 

\subsection{Concluding remarks} \label{remark}

There is a connection between the method of our proof and Fernique's  chaining argument  \cite{Fer1971} (see  \cite{talagrand1996majorizing} for a survey).
The goal of the chaining method is to bound the supermum $\sup_{t \in B} X_t$ where $B$ is a domain in a metrics space and $X_t$ is a Gaussian process. 
In this case,  the bad event $Bad (v)$ can roughly be defined as $X_v \ge M_v$, for some candidate value $M_v$. One then considers a chain of sets 
in order to bound $\P( \cup_{v \in B} Bad (v)) $.  This, in spirit, is similar to the purpose of  Lemma \ref{newnet}. 

After  this,  the arguments become different in all aspects. First,  in our setting, the bad event $Bad(v)$ can have any nature. 
Next,  in the chaining argument, the sets $\CN_j$  are defined using the metrics of $B$, while in our case, it is crucial to use 
a different metrics. We construct  $\CN_j$ using the $l_{\infty}$ norm, rather than the natural  $l_1$ norm used to define the domain $B$. Finally, in the chaining case it is easy to bound  
$\P( Bad (u) \backslash Bad(v)) $,   using  the fact that  $\P( |X_u- X_v| \ge t) \le 2 \exp( - \frac{t^2}{dist (u,v)^2 }) $, which is the basic property of a Gaussian process. 
In our case,  bounding $\P( Bad (u) \backslash Bad(v)) $ is an essential step (Lemma \ref{hardlemma}), which requires the development of the refined Bernstein's inequality.

%The proof from above is then easily modified to handle this situation and the verifications in Section \ref{magnitude} will require $p \geq C n \log^{4} n$.

\section{ The algorithm and concentration of random matrices} \label{algorithms}

As the algorithm and analysis are discussed 
extensively in \cite{SWW}, we will be brief and the readers can consult \cite{SWW} for more details. 
\cite{SWW} introduces the dictionary learning
algorithm ER-SpUD.  The key insight in the design of ER-SpUD  
is that the rows of $\x$ are likely to be the sparsest vectors in the row space of $\y$.  (This 
observation also appeared   \cite{zibulevsky2000blind} and \cite{mairal2009online}.)
\cite{SWW} proposed  to find these vectors by considering the following optimization problems.

$$
\text{minimize } \| {w^T} \y \|_1 \text{ subject to } {r^T w} = 1
$$
where ${r}$ is a row of two columns of $\y$.

Using  $l_1$ optimization for finding sparse vectors is a natural idea, and the authors of \cite{SWW} pointed out that 
such an approach was already proposed in \cite{plumbley2007dictionary}
and \cite{jaillet2010l1}. The difference is the new constraint  $r^Tw =1$. (Earlier works used different constraints.)

By a change of variables ${z = A^T w}$, ${b = A^{-1} r}$, we can consider the equivalent problem
\begin{equation} \label{optimization}
\text{minimize } \| {z^T} \x \|_1 \text{ subject to } {b^T z} = 1.
\end{equation}

The algorithm presented in \cite{SWW} is outlined below  
(for those familiar with \cite{SWW}, note that we are presenting the two-column version of ER-SpUD):
\begin{algorithm}[H]
\caption{ER-SpUD}
  \begin{algorithmic}[1]
     \State Randomly pair the columns of ${Y}$ into $p/2$ groups $g_j = \{ {Y}e_{j_{1}}, {Y} e_{j_{2}}\}$
     
     \State For $j=1, \dots , p/2$ \newline
      $\text{     }$ Let ${r}_j ={Y}e_{j_{1}} + {Y}e_{j_{2}}$, where $g_j = \{ {Y}e_{j_{1}}, {Y} e_{j_{2}}\}$  \newline 
      $\text{     }$   Solve $min_{{w}} \|{w^T} \y \|_1$ subject to $(\y {r_j})^T {w} = 1$, and set ${s_j} = 
                       {w^T} \y$.
     \State Use Greedy algorithm to reconstruct $\x$ and ${A}$. 
  \end{algorithmic}
\end{algorithm}

\begin{algorithm}[H]
\caption{Greedy}
   \begin{algorithmic}[1]
     \State Require: $S = \{s_1,\dots, s_T\} \subset \mathbb{R}^p$
     \State For $i = 1 \dots n$ \newline
                    $\text{     }$ REPEAT \newline
                    $\text{     }$ $\text{     }$ $l \gets arg \text{ } min_{{s_l} \in S}  \|{s_l}\|_0$, breaking ties arbitrarily \newline
                    $\text{     }$ $\text{     }$ ${x_i} = {s_l}$ \newline
                    $\text{     }$ $\text{     }$ $S = S \backslash \{ {s_l} \} $ \newline
                    $\text{     }$ UNTIL $rank([{x_1}, \dots , {x_i}]) = i$
     \State Set $\x = [{x_1}, \dots , {x_i}]^T$, and ${A} = \y \y^T (\x \y^T)^{-1}$
    \end{algorithmic}
\end{algorithm}

%We sketch how the following lemma, whose proof is of independent interest and sketched in \ref{proofnew}, can be applied to yield the near-optimal value for $p$.  

A key technical step in analyzing ER-SpUD is the following lemma, which asserts that if $p$ is sufficiently large, then with high probability $\| X^T v \|_1$ is close to its mean, 
simultaneously for all unit vectors $v \in \R^n$. 

\begin{lemma} \label{old} 

For every constant $1 \ge \delta >0$ there is a constant $C_0 >0$ such that the following holds. 
If $\theta \geq \frac{1}{n}$ and $p  \ge C_0 n^2 \log^2  n $, then with  probability $1 - o(1) $, for all $v \in \R^n$ 

\begin{equation}  \label{newmain}|   \|X^T v  \| _1 - \E \|  X^T v  \|_1 |  \le  \delta  \E \| X^T v \|_1 . \end{equation} 
\end{lemma}

%This is the only place where  the bound $p \ge C_0 n^2 \log n$ is critical in the proof of Theorem \ref{theorem:upper} in \cite{SWW}. 

This lemma appears implicitly in \cite{SWW}. Dan Spielman pointed out to us that this would imply the critical \cite[Lemma 17]{SWW}. The bound $p \ge C n^2 \log^2  n$ 
is of importance in the proof of this lemma.

Our Theorem \ref{theorem:main}, which pushes $p$ to $C n \log^{4} n$, is an improved version of Lemma \ref{old}.

\vskip2mm 

With Theorem \ref{theorem:main} in hand, let us now sketch the proof of Theorem \ref{theorem:main1}, following the analysis in \cite{SWW}. 

Notice that if the solution of the $l_1$ optimization problem, $z_*$, is 1-sparse, then the algorithm will recover a row of $X$.   
The proof of the theorem relies on showing that $ z_*$, is supported on the non-zero indices of $ b$ and that with high-probability, $ z_*$ is in fact 1-sparse.  The first goal allows us to focus our attention on a submatrix of $\x$ which will be convenient for technical reasons.  To 
address this first issue, we prove the following.

\begin{lemma} \label{SpLem}
Suppose that $\x$ satisfies the Bernoulli-Subgaussian model.  There exists a numerical constant $C>0$ such that if $\theta n \geq 2$ and
$$
p>Cn \log^{4} n
$$
then the random matrix $\x$ has the following property with probability at least $1-o(1)$.

(\textbf{P1}) For every {b} satisfying $\| {b} \|_0 \leq 1/8\theta$, any solution ${z_*}$ to the optimization 
problem \ref{optimization}
has $supp({z_*}) \subseteq supp({b})$.
\end{lemma}

\textit{ Sketch of the Proof of Lemma \ref{SpLem}.} 
We let $J$ be the indices of the $s$ non-zero entries of
${b}$.  Let $S$ be the indices of the nonzero columns in $\x_J$, and let ${z_0} = {P_J} {z_*}$ (the restriction to those coordinates indexed by $J$).  Define ${z_1} = {z_*}-{z_0}$.  We demonstrate that ${z_0}$ has at least as low an objective as ${z_*}$
 so ${z_1}$ must be zero.  One can show using the triangle inequality that
 $$
 \| {z_*^T} \x \|_1 \geq \|{z_0^T} \x \|_1 - 2 \| {z_1^T} \x^S \|_1 + \| {z_1^T} \x \|_1. 
 $$
Thus, if $\| {z_1^T} \x \|_1 - 2 \| {z_1^T} \x^S \|_1>0$, then ${z_0}$ has a lower objective value.
We need this inequality to hold for all ${z}$ with high probability.
Notice that 
$$
\E[ \| {z^T} \x \|_1 - 2 \| {z^T} \x^S \|_1 ] = (p-2|S|) \E | {z^T} \x_1 |
$$
It is easy to show that $|S| < p/4$ with high probability so $(p-2|S|) > 0$ with high probability.
Therefore, if we can show that $\| {z^T} \x \|_1 - 2 \| {z^T} \x^S \|_1$ is concentrated near its positive expectation we
are done.  

We see that it suffices to show the result for the worst case $|S| = p/4$.  Now we make critical use of  Theorem
\ref{theorem:main}, which asserts that with high probability,
$$
\| {z^T} \x \|_1 \geq \frac{5}{8} \E\| {z^T} \x \|_1 = 
\frac{5p}{8} \E | {z^T} \x_1 |.
$$   
and
$$
\| {z^T} \x^S \|_1 \leq \frac{1}{2} \E \|{z^T} \x^S \|_1 = \frac{p}{8} \E | {z^T} \x_1 |.
$$
so
$$
\| {z^T} \x \|_1 - 2 \| {z^T} \x^S \|_1 \geq \frac{p}{2}\E | {z^T} \x_1 | >0.
$$

Having proved Lemma \ref{SpLem}, the rest of  the proof is relatively simple and follows \cite{SWW} exactly.  
The success of the algorithm now depends on the existence of a sufficient gap between the largest and second largest entry in ${b}$.  The intuition is that if $\x$ preserved
the $l_1$ norm exactly, i.e. $\| {z^T} \x \|_1 = c \|z \|_1$, then the minimization procedure will output the vector ${z}$ of smallest $l_1$ norm  such that 
${b^T z} =1$, which is just ${e_{j_*}} / b_{j_*}$, where $j_*$ is the index of the element of ${b}$ with the largest magnitude.  However, $\x$ only preserves
the $l_1$ norm in an approximate sense.  Yet, the algorithm will still extract a column of $\x$ if there is a significant gap between the largest element of ${b}$
and the second largest.

\section{Rectangular dictionaries and  Theorem \ref{theorem:main2} } \label{appendix:rectangular}

We now present a generalization of ER-SpUD, which enables us to deal with rectangular dictionary. 
Consider a full rank matrix $A$ of size $n > m$, such that $n > m $, and the equation $AX= Y$. To deal with this setting, we 
first augment $A$ to be a square, $n \times n$,  invertible matrix.  Of course, the issue is that 
one does  not know $A$,  and also need to figure out  how the augmentation  changes the product $Y$. 

We can solve this issue using a random augmentation. For instance, we can use   $n \times (n-m)$ gaussian matrix $B$ to  augment  $A$ to a square matrix $A'$
(the entries  in $B$ are iid standard gaussian). It is trivial that the augmented matrix has full rank with probability 1, since the probability that a gaussian vector belongs to any fixed hyperplane is zero. We can also augment $\x$ from an $m \times p$ matrix to a $n \times p$ matrix, $\x'$ by an $(n-m) \times p$ random matrix $Z$ with entries iid to those  of $\x$.  This augmentation process yields a matrix equation 

$$Y' = A' X' $$

\noindent  where  $\y' = \y + {E}$ where ${E} = BZ$  (Figure \ref{fig:rectangular}).  In practice, we can first generate $B, Z$, then compute $E := BZ$ and   construct  $\y' := Y+ E'$. Next 
then apply the ER-SpUD algorithm to  the equation $Y' = A' X'$  to  recover ${A}'$ and $\x'$ with high probability.  From these two
matrices, we can then deduce ${A}$ and $\x$.

Using a gaussian (or any continuous) augmentation is convenient, as the resulting matrix is obviously full rank.
However, it is, in some way, a cheat. Apparently, a gaussian number does not have any finite representation, thus it takes forever to read the input, let 
alone process it. A common practice is to truncate (as a matter of fact, the computer only generates a finite approximation of the 
gaussian numbers anyway), and hope that the truncation is fine for our purpose. But then we face a non-trivial  theoretical question 
to analyze this approximation. How many decimal places are enough ?  Even if we can prove a guarantee here, using it in practice 
would require computing with a matrix with many long entries, which significantly  increases the running time.

We can avoid this problem by using  random matrices with discrete distributions, such as $\pm 1$.  
The technical issue now is to prove the full rank property.
This is a highly non-trivial problem,but luckily was taken care of in the following result of Bourgain, Vu, and Wood \cite{BVW}. 

\begin{theorem} \label{BVW} 
For every $\epsilon > 0$ there exists $\delta > 0$ such that the following holds.  Let 
$N_{f,n}$ be an $n$ by $n$ complex matrix in which $f$ rows contain fixed, non-random entries
and where the other rows contain entries that are independent discrete random variables.  If the 
fixed rows have co-rank $k$ and if for every random entry $\alpha$, we have $max_x \Prob(\alpha = x)
\leq 1-\epsilon$, then for all sufficiently large $n$
$$
\Prob(N_{f,n} \text{ has co-rank} > k) \leq (1-\delta)^{n-f}.
$$
\end{theorem}
Letting, $k=0$ and $f=m$, the result shows that if we augment ${A}$ by 
$n \times (m-n)$ random Bernoulli matrix, this new matrix, 
${A}'$, will be nonsingular with high probability, given that $n-m =\omega(1) $. 

\begin{figure}[h] 
    \centering
    \includegraphics[width=\textwidth]{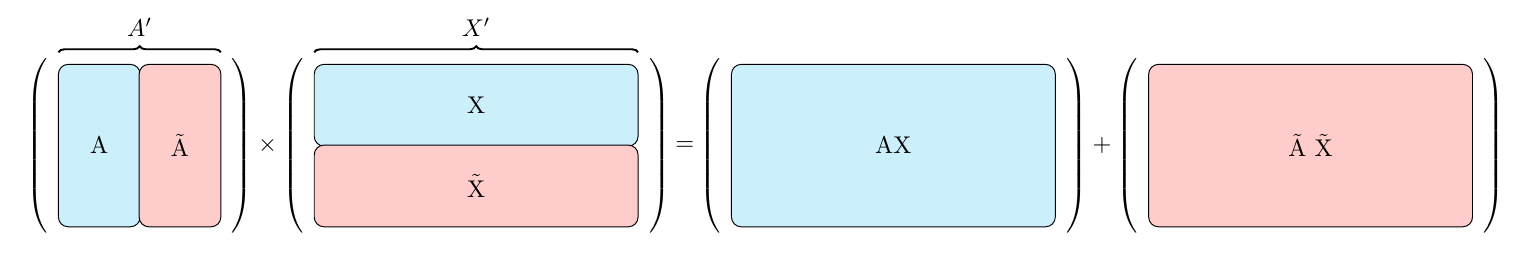}
    \caption{Rectangular A with $n > m$}
    \label{fig:rectangular}
\end{figure}
We summarize our reasoning in the following algorithm.
\begin{algorithm}[H]
\caption{Rectangular Algorithm}
  \begin{algorithmic}[1]
     \State Generate a $(n-m) \times p$ matrix $Z$ with iid random variables that agree with 
            the model for $X$.  
     \State Generate a $n \times (n-m)$ matrix $B$ with iid entries (either Gaussian or Rademacher).    
     \State Run {ER-SpUD} on $Y' = Y + BZ$
     \State Remove the rows of $A'$ and the columns of $X'$ from the output of ER-SpUD.
  \end{algorithmic}
\end{algorithm}

\section{Optimal bound for  very sparse random matrices} \label{verysparse}

In this section, we discuss Theorem \ref{theorem:main3}.
We present a simple algorithm (see below) and use this algorithm to prove Theorem \ref{theorem:main3}, obtaining the optimal bound 
$p = C n \log n$. 

\begin{algorithm}[H] \label{Alg:sparse}
\caption{Very-sparse Algorithm}
  \begin{algorithmic}[1]
     \State Partition the columns of Y into a minimum number of groups $G_i$ whose members are multiples of each other. 
     \State Choose representatives of those $G_i$ with more than two members to be the columns of A up to scaling.
  \end{algorithmic}
\end{algorithm}

\textit{Proof of Theorem \ref{theorem:main3}.}   
Since $A$ is nonsingular, any two columns of $Y$ that are multiples of each other must be linear 
combinations of the same columns of $A$.  For a group $G_i$ to have more than two members
would require that there be more than two columns in $X$ with their non-zero entries in the same
rows.    
\begin{definition}
We say that a set of columns are {aligned} if they each have 
more than one nonzero entry and their non-zero entries occur in the same positions.
\end{definition}

\begin{lemma}\label{very sparse}
The probability that $X$ has more than two aligned columns is $o(1)$.
\end{lemma}    
Thus, the algorithm is likely to yield only columns of $A$.  We now need to show that
all the columns of $A$ will be outputted with
high probability.  
\begin{definition}
We say the column $\mathbf{a}$ of $A$ is {k-represented} 
if some group $G_i$ consists of multiples of 
$\mathbf{a}$ and $|G_i| = k$.  In particular, if no multiple of the $j$th column, $\mathbf{a_j}$, shows up in the 
columns of $Y$ then $\mathbf{a_j}$ is $0$-represented.
A column is {well represented} if it is
$k$-represented for $k>2$.  \end{definition}

Notice that the algorithm will output a multiple of every column that 
is well represented.

The following lemma finishes the proof of Theorem 
\ref{theorem:main3}.

\begin{lemma}\label{coupon}
The probability that every column $\mathbf{a_i}$ is well represented is $1-o(1)$.
\end{lemma} 

%The proofs of Lemmas \ref{very sparse} and \ref{coupon} are in Appendix \ref{technical}.

\subsection{Proofs of Sparse Algorithm} \label{technical}
\textit{Proof of Lemma \ref{very sparse}.}
Given the choice of $\theta$, we know that 
the number of nonzero entries in any column of $X$ will converge to the Poisson distribution.
We ignore the $o(1/n)$ error terms from this approximation in later calculations to alleviate clutter.   
To calculate the probability, we
condition on the number of
nonzero entries, and then we bound 
the probability that three specific columns 
have the required property, and finally we use the 
union bound.  This yields an upper bound of 
$$
  {n \choose 3}\sum_{k \geq 2} \frac{e^{-3c}}{(k!)^3} \frac{1}{{n \choose k}^2} = o(1)
$$
\QED
\newline
\newline
\textit{Proof of Lemma \ref{coupon}.}
By the union bound, 
$$
\Prob(\exists i \text{ such that } \mathbf{a_i} \text{ is not well represented}) \leq n \Prob(\mathbf{a_1} \text{ is not well represented})
$$
Partitioning into disjoint events yields
$$
\Prob(\mathbf{a_1} \text{ is not well represented}) = \sum_{j=0}^2\Prob(\mathbf{a_1} \text{ is } \text{$j$-represented})
$$
Notice that a multiple of $\mathbf{a_1}$, say  $a*\mathbf{a_1}$, appears as a column of $Y$ if and only if $a*\mathbf{e_1} = (a,0,0, \dots ,0)^T$, with $a \neq 0$,
is $X^j$, the $j$th column of $X$, for some $j$.
Now, using the Poisson approximation we can bound each term in the summand.
For example, for the probability of being $0$-represented, we can divide into the case
that $X^i$ does not have exactly one non-zero element and the case that $X^i$ has exactly one
non-zero term but not in the first row.  We use $C$ to indicate an absolute constant which may
change with each appearance.  
$$
\Prob(\mathbf{a_1} \text{ is } \text{$0$-represented}) \leq \left( (1 - c e^{-c}) + e^{-c} \frac{n-1}{n} \right)^p \leq C \exp(-Cp/n)
$$
Similarly,
$$
\Prob(\mathbf{a_1} \text{ is } \text{$1$-represented}) \leq n \left( \frac{c e^{-c}}{n} \right)
 \left( (1 - c e^{-c}) + e^{-c} \frac{n-1}{n} \right)^{p-1} \leq C \exp(-Cp/n)
$$
and
$$
\Prob(\mathbf{a_1} \text{ is } \text{$2$-represented}) \leq {n \choose 2} \left( \frac{c e^{-c}}{n} \right)^2
 \left( (1 - c e^{-c}) + e^{-c} \frac{n-1}{n} \right)^{p-2} \leq C \exp(-Cp/n)
$$
Thus, 
$$
\Prob(\mathbf{a_1} \text{ is not well represented}) \leq C \exp(\log n - Cp/n) = o(1)
$$
for $p = C' n \log n$ for a large enough $C'$.
\QED

\section{Proof of Theorem \ref{theorem:main4} } \label{theoreticalbound} 

\subsection{Lemmas Independent of Symmetry}
We first state the necessary lemmas from \cite{SWW} whose proofs do not use the symmetry of the random variables.

\begin{lemma} \label{rowspace}
If $rank(X)=n$, $A$ is nonsingular, and $Y$ can be decomposed into $Y=A'X'$, then the row spaces of $X'$, $X$, and $Y$ are the same.
\end{lemma}

The general idea is to show that the sparsest vectors in the row-span of $Y$ are the rows of $X$. 
Since all of the rows of $X'$ lie in the row-span of $Y$, intuitively, they can be sparse only when 
they are multiples of the rows of $X$.  Naively, this is because rows of $X$ are likely to have
nearly disjoint supports.  Thus, any linear combination of them will probably increase the 
number of nonzero entries.

\begin{lemma}
Let $\Omega$ be an $n \times p$ Bernoulli($\theta$) matrix with $1/n < \theta < 1/4$.  For each set
$S \subseteq [n]$, let $T_S \subseteq [p]$ be the indices of the columns of $\Omega$ that have at least
one non-zero entry in some row indexed by $S$.
\begin{enumerate}[(a)]
\item For every set $S$ of size 2,
$$
\Prob(|T_S| \leq (4/3)\theta p) \leq \exp\left(-\frac{\theta p}{108}\right)
$$
\item For every set $S$ of size $\sigma$ with $3 \leq \sigma \leq 1/\theta$,
$$
\Prob(|T_S| \leq (3\sigma/8)\theta p) \leq \exp\left(-\frac{\sigma \theta p}{64}\right)
$$
\item For every set $S$ of size $\sigma$ with $1/\theta \leq \sigma$,
$$
\Prob(|T_S| \leq (1-1/e)p/2) \leq \exp\left(-\frac{(1-1/e)p}{8}\right)
$$
\end{enumerate}
\end{lemma}

\subsection{Generalized Lemmas}
We will use a result of \cite{rudelson2008littlewood}.
\begin{lemma} \label{rudelson}
Let $\xi_1, \dots, \xi_n$ be independent centered random variables with variances at least $1$ and fourth moments bounded by $B$.  Then there exists $\nu \in (0,1)$ depending only on $B$, such that for
every coefficient vector $a = (a_1, \dots, a_n) \in S^{n-1}$ the random sum $S = \sum_{k=1}^n a_k \xi_k$
satisfies
$$
\Prob (|S| < \frac{1}{2}) \leq \nu
$$
\end{lemma}

\begin{definition}
We call a vector $\alpha \in \mathbb{R}^n$ fully dense if for all $i \in [n]$, $\alpha_i \neq 0$.
\end{definition}
\begin{lemma}\label{generalized}
For $b>s$, let $H \in \mathbb{R}^{s \times b}$ be a matrix with one nonzero in each column.  Let $R$ be a s-by-b matrix with independent centered random variables with 
variances at least $1$ and bounded fourth moments.  
Define $U = H \odot R$ Then the probability that the left nullspace of $U$ contains a fully dense vector is at most 
\end{lemma}
\textit{ Proof of Lemma \ref{generalized}.}  Let $U = [u_1|\dots|u_b]$ denote the columns of $U$ and 
for each $j \in [b]$, let $N_j$ be the left nullspace of $[u_1|\dots|u_j]$.  We show that with high 
probability $N_b$ cannot contain a fully dense vector.
This can be done by showing that if $N_{j-1}$ contains a fully dense vector then with probability 1/2 the dimension of $N_j$ is less than the dimension of $N_{j-1}$.  Formally, consider a fully dense
vector $\alpha \in N_{j-1}$.  If $u_j$ contains only one nonzero entry, then $\alpha^T u_j \neq 0$ 
reducing the dimension of $N_j$.  If $u_j$ contains more than one non-zero entry, then Lemma
\ref{rudelson} implies that the probability, over the choice of entries of $R_j$, that $\alpha^T u_j = 0$ is less than $1/2$.   

Note that the dimension cannot decrease more than $s$ times.  For $N_b$ to contain a fully dense vector, there must be at least $b-s$ columns for which the dimension of the nullspace does nto decrease.  Let $F \subset [b]$ have size $b-s$.  The probability that for every $j \in F$, $N_{j-1}$ contains a fully dense vector and that the dimension of $N_j$ equals the dimension of $N_{j-1}$ is at most $2^{-b+s-1}$.  By the union bound, the probability that $N_b$ contains a fully dense vector is at most 
$$
{b \choose {b-s}}2^{-b+s} \leq \left(\frac{eb}{s} \right)^s 2^{-b+s} \leq 2^{-b+s \log(e^2b/s)}
$$ 
\QED

The proofs of the following lemmas are identical to those in \cite{SWW} except that they now use our
more general Lemma \ref{generalized} along with the lemmas in the previous section.
\begin{lemma}
For $t>200s$, let $\Omega \in  \{0,1\}^{s \times t}$ be any binary matrix with at least one nonzero 
in each column.  Let $R \in \mathbb{R}^{s \times t}$ be a random matrix whose entries are iid random variables, with $\Prob(R_{ij}=0)=0$, and let $U=\Omega \odot R$.  Then, the probability that there exists a fully-dense vector $\alpha$ for which $\|\alpha^T U\|_0 \leq t/5$ is at most $2^{-t/25}$.
\end{lemma}

\begin{lemma} \label{Sp7}
If $X = \Omega \odot R$ follows the Bernoulli-Subgaussian model with $\Prob(R_{ij}=0)=0$, $1/n < \theta < 1/C$ and $p > C n \log n$, then the probability that there is a vector $\alpha$ with support of size larger than $1$ for which 
$$
\|\alpha^T X\|_0 \leq (11/9) \theta p
$$
is at most $\exp(-c\theta p)$, and $C,c$ are numerical constants.  
\end{lemma}

\subsection{Proof of Theorem \ref{theorem:main4}}
Say $Y$ can be decomposed as $A' X'$.
From Lemma \ref{Sp7}, we know that with probability at most $\exp(-c \theta p)$, any linear combination of two or more rows of $X$ has at least $(11/9)  \theta p$ nonzeros.  
By a simple Chernoff bound, the probability that any row of $X$ has more than $(10/9) \theta p$
nonzero entries is bounded by $n \exp(-\theta p/243)$.  Thus, the rows of $X$ are likely the sparsest in $row(X)$.

On the previous event of probability at least $1-\exp(-c\theta p)$, $X$ does not have any left null vectors with more than one nonzero entry.  Therefore, if the rows of $X$ are nonzero, $X$ will have no nonzero vectors in its left nullspace.  The probability that all of the rows of $X$ are nonzero
is at least $1-n(1-\theta)^p \geq 1-n\exp(-cp)$.  From this, by Lemma \ref{rowspace}, we get $row(X) = row(Y)=row(X')$. 
Hence, we can conclude that every row in $X'$ is a scalar multiple of a row of $X$.
\QED

%{\bf Kyle: I change represented to well represented; you need to change the proof accordingly.} 

\section{Numerical Simulations} \label{numerical}
We demonstrate that the efficiency of the ER-SpUD algorithm is not improved with larger $p$ values 
beyond the threshold conjectured.  In Figure \ref{fig:erspud}, we  have chosen 
$A$ to be an $n \times n$ matrix of independent $N(0,1)$ random variables.  The $n \times p$ matrix 
$X$ has $k$ randomly chosen non-zero entries which are Rademacher.  The graph on the left of Figure \ref{fig:erspud}
is generated with $p = 5 n \log n$ and the one on the right with $p = 5 n^2 \log^2 n$.  For both graphs,
$n$ varies from $10$ to $60$ and $k$ from $1$ to $10$.  
Accuracy is measured in terms of relative error: 
$$
re(A',A) = min_{\Pi,\Lambda} \|A'\Lambda \Pi - A \|_F / \|A\|_F
$$
The average relative error over ten trials is reported.

\begin{figure}[h]
    \centering
    \includegraphics[scale = .6]{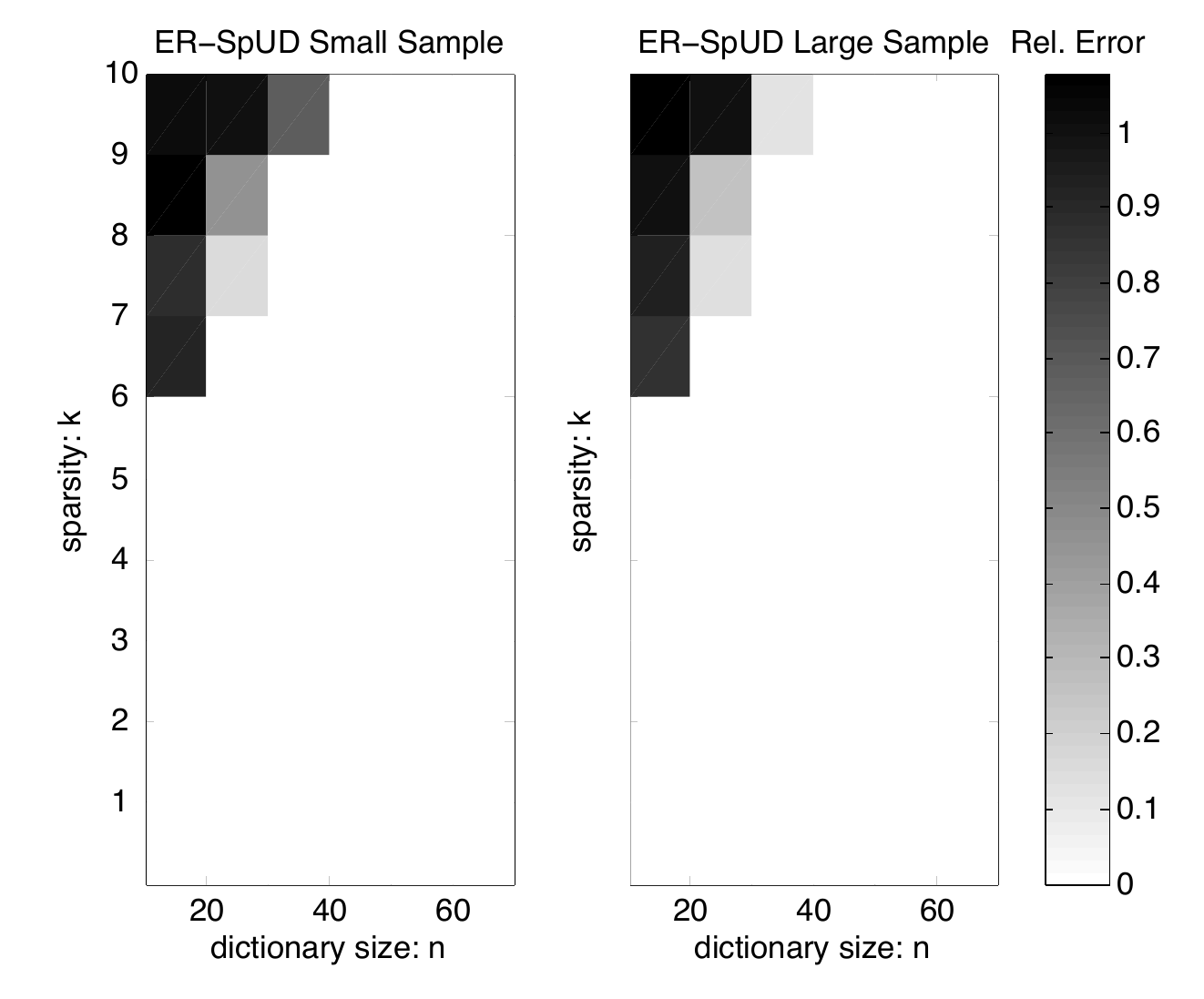}
    \caption{Mean relative errors of ER-SpUD with $p =5n \log n$ versus $p = 5 n^2 \log^2 n$}
    \label{fig:erspud}
\end{figure}

We then ran our Algorithm \ref{Alg:sparse} in a sparse regime to compare its performance with that of ER-SpUD (see Figure \ref{fig:comparison}.  $A$ was as before, but since our algorithm
relies on the appearance of 1-sparse columns in $X$, we cannot fix sparsity as in our first experiments.
Rather, we vary the Bernoulli parameter $\theta$ from $0.02$ to $0.18$, and the $\chi_{ij}$ are Rademacher.
One can see the expected phase transition at which point the matrix $X$ is no longer
sparse enough for our algorithm.  In the regime for which the algorithm was designed, the relative error of our output is on the 
same order as that of ER-SpUD.  Furthermore, our algorithm runs much quicker and has no trouble with inputs of size up to $n=500$. (The numerical experiments were completed on a Macbook Pro.)

Finally, we compare the outcome of our optimal $p$ value with that of a much larger sample size ($p = O(n^2 \log^2 n)$). We let $n$ range from $10$ to $200$ and $\theta$ from $0.01$ to $0.08$.  Figure \ref{fig:sparse}
shows that the efficacy 
of the algorithm is not much improved despite the dramatic increase in $p$.  The threshold for failure is identical.     
\begin{figure}[h]
    \centering
    \includegraphics[scale = .7]{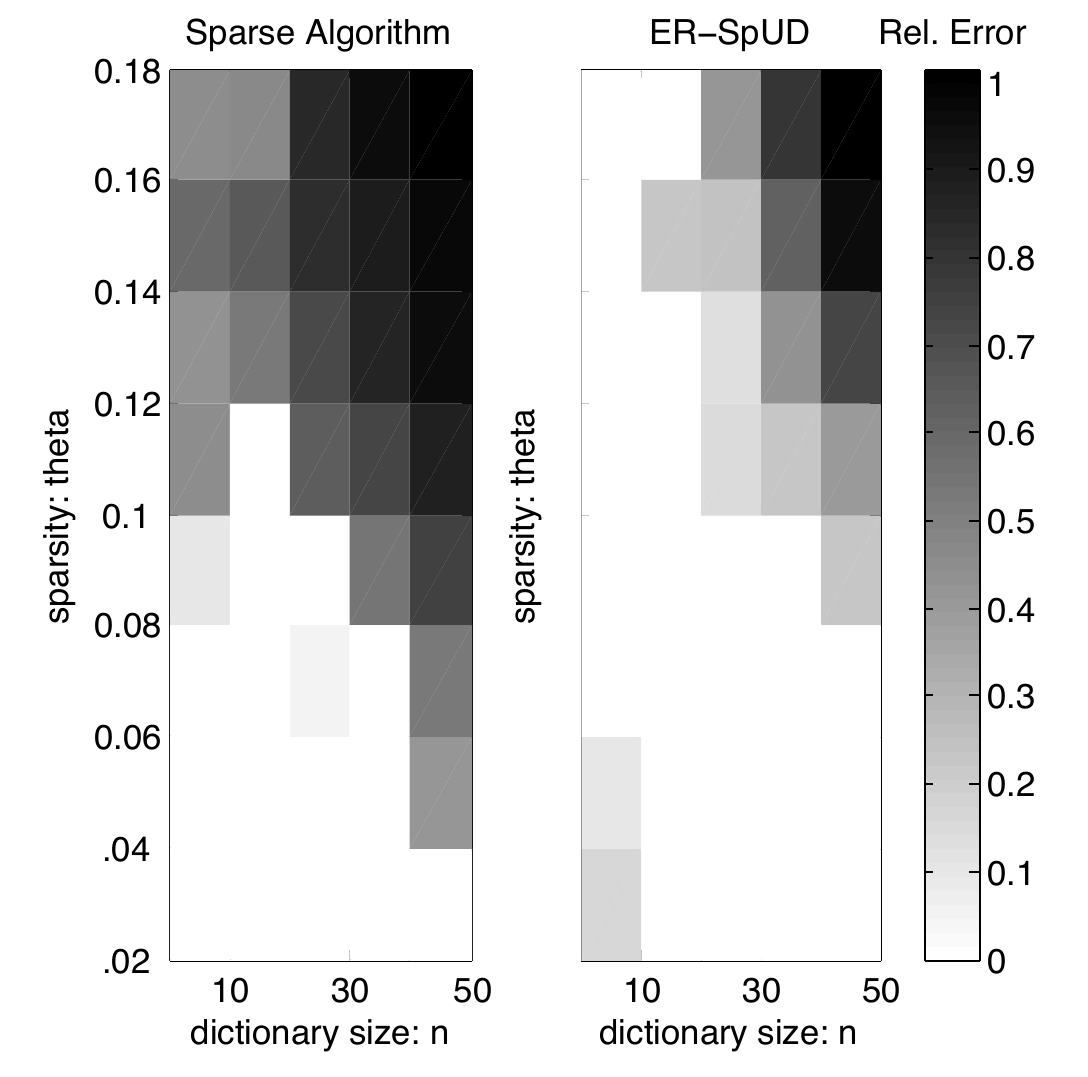}
    \caption{Mean relative errors with varying sparsity $\theta$.  Here, $p=5 n \log n$.}
    \label{fig:comparison}
\end{figure}

\begin{figure}[h]
    \centering
    \includegraphics[scale = .6]{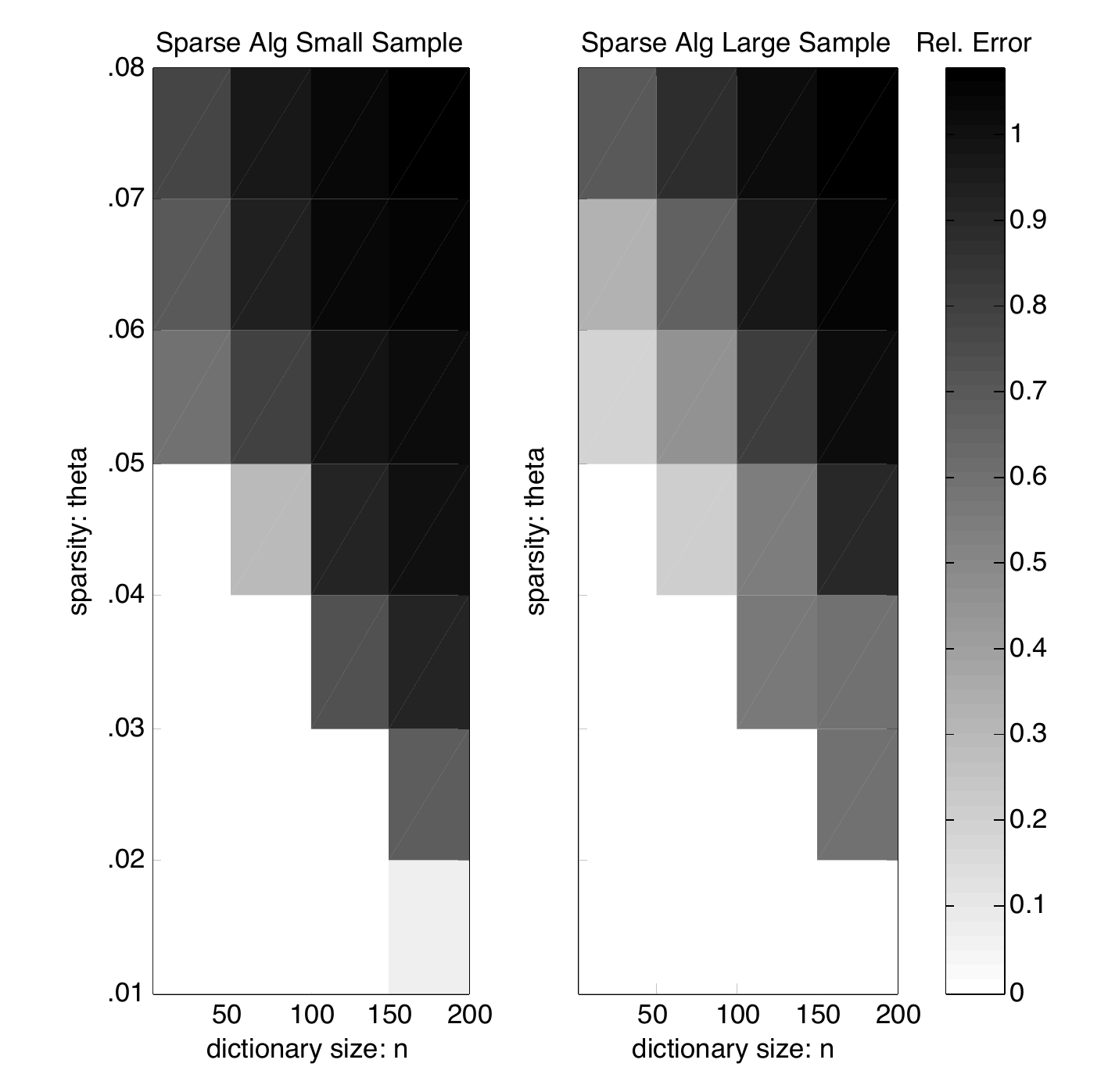}
    \caption{Mean relative errors of Algorithm \ref{Alg:sparse} with $p =5n \log n$ versus $p = 5 n^2 \log^2 n$}
    \label{fig:sparse}
\end{figure}

\bibliographystyle{plain}
\bibliography{Journalversion}

\begin{thebibliography}{10}

\bibitem{aharon2006img}
Michal Aharon, Michael Elad, and Alfred Bruckstein.
\newblock The k-svd: An algorithm for designing overcomplete dictionaries for
  sparse representation.
\newblock {\em Signal Processing, IEEE Transactions on}, 54(11):4311--4322,
  2006.

\bibitem{aharon2006uniqueness}
Michal Aharon, Michael Elad, and Alfred~M Bruckstein.
\newblock On the uniqueness of overcomplete dictionaries, and a practical way
  to retrieve them.
\newblock {\em Linear algebra and its applications}, 416(1):48--67, 2006.

\bibitem{BVW}
Jean Bourgain, Van~H Vu, and Philip~Matchett Wood.
\newblock On the singularity probability of discrete random matrices.
\newblock {\em Journal of Functional Analysis}, 258(2):559--603, 2010.

\bibitem{engan1999method}
Kjersti Engan, Sven~Ole Aase, and J~Hakon~Husoy.
\newblock Method of optimal directions for frame design.
\newblock In {\em Acoustics, Speech, and Signal Processing, 1999. Proceedings.,
  1999 IEEE International Conference on}, volume~5, pages 2443--2446. IEEE,
  1999.

\bibitem{Fer1971}
X~Fernique.
\newblock Regularite de processus gaussien.
\newblock In {\em Invent Math.}, pages 304--321. 1971.

\bibitem{georgiev2004blind}
Pando Georgiev, Fabian Theis, and Andrzej Cichocki.
\newblock Blind source separation and sparse component analysis of overcomplete
  mixtures.
\newblock In {\em Acoustics, Speech, and Signal Processing, 2004.
  Proceedings.(ICASSP'04). IEEE International Conference on}, volume~5, pages
  V--493. IEEE, 2004.

\bibitem{gottlieb2010matrix}
Lee-Ad Gottlieb and Tyler Neylon.
\newblock Matrix sparsification and the sparse null space problem.
\newblock In {\em Approximation, Randomization, and Combinatorial Optimization.
  Algorithms and Techniques}, pages 205--218. Springer, 2010.

\bibitem{jaillet2010l1}
Florent Jaillet, R{\'e}mi Gribonval, Mark~D Plumbley, and Hadi Zayyani.
\newblock An l1 criterion for dictionary learning by subspace identification.
\newblock In {\em Acoustics Speech and Signal Processing (ICASSP), 2010 IEEE
  International Conference on}, pages 5482--5485. IEEE, 2010.

\bibitem{kreutz2003dictionary}
Kenneth Kreutz-Delgado, Joseph~F Murray, Bhaskar~D Rao, Kjersti Engan, Te-Won
  Lee, and Terrence~J Sejnowski.
\newblock Dictionary learning algorithms for sparse representation.
\newblock {\em Neural computation}, 15(2):349--396, 2003.

\bibitem{li2013discriminative}
Liangyue Li, Sheng Li, and Yun Fu.
\newblock Discriminative dictionary learning with low-rank regularization for
  face recognition.
\newblock In {\em Automatic Face and Gesture Recognition (FG), 2013 10th IEEE
  International Conference and Workshops on}, pages 1--6. IEEE, 2013.

\bibitem{mairal2009online}
Julien Mairal, Francis Bach, Jean Ponce, and Guillermo Sapiro.
\newblock Online dictionary learning for sparse coding.
\newblock In {\em Proceedings of the 26th Annual International Conference on
  Machine Learning}, pages 689--696. ACM, 2009.

\bibitem{olshausen1996emergence}
Bruno~A Olshausen et~al.
\newblock Emergence of simple-cell receptive field properties by learning a
  sparse code for natural images.
\newblock {\em Nature}, 381(6583):607--609, 1996.

\bibitem{plumbley2007dictionary}
Mark~D Plumbley.
\newblock Dictionary learning for l1-exact sparse coding.
\newblock In {\em Independent Component Analysis and Signal Separation}, pages
  406--413. Springer, 2007.

\bibitem{rudelson2008littlewood}
Mark Rudelson and Roman Vershynin.
\newblock The littlewood--offord problem and invertibility of random matrices.
\newblock {\em Advances in Mathematics}, 218(2):600--633, 2008.

\bibitem{SWW}
Daniel~A Spielman, Huan Wang, and John Wright.
\newblock Exact recovery of sparsely-used dictionaries.
\newblock In {\em Proceedings of the Twenty-Third international joint
  conference on Artificial Intelligence}, pages 3087--3090. AAAI Press, 2013.

\bibitem{talagrand1996majorizing}
Michel Talagrand.
\newblock Majorizing measures: the generic chaining.
\newblock {\em The Annals of Probability}, pages 1049--1103, 1996.

\bibitem{wright2009robust}
John Wright, Allen~Y Yang, Arvind Ganesh, Shankar~S Sastry, and Yi~Ma.
\newblock Robust face recognition via sparse representation.
\newblock {\em Pattern Analysis and Machine Intelligence, IEEE Transactions
  on}, 31(2):210--227, 2009.

\bibitem{zhang2010discriminative}
Qiang Zhang and Baoxin Li.
\newblock Discriminative k-svd for dictionary learning in face recognition.
\newblock In {\em Computer Vision and Pattern Recognition (CVPR), 2010 IEEE
  Conference on}, pages 2691--2698. IEEE, 2010.

\bibitem{zibulevsky2003blind}
Michael Zibulevsky.
\newblock Blind source separation with relative newton method.
\newblock In {\em Proc. ICA}, volume 2003, pages 897--902, 2003.

\bibitem{zibulevsky2000blind}
Michael Zibulevsky and Barak~A Pearlmutter.
\newblock Blind source separation by sparse decomposition.
\newblock In {\em AeroSense 2000}, pages 165--174. International Society for
  Optics and Photonics, 2000.

\end{thebibliography}

\appendix

\end{document}